\documentclass[draft]{amsart}

\newcommand{\df}{\dfrac}

 \renewcommand{\a}{\alpha}

\newcommand{\e}{\epsilon}
\renewcommand{\d}{{\delta}}

\renewcommand{\l}{\lambda}

\renewcommand{\k}{\chi}

\renewcommand{\t}{\varphi}

\renewcommand{\i}{\infty}
\renewcommand{\b}{\beta}

\newcommand{\beqs}{\begin{equation*}}
\newcommand{\eeqs}{\end{equation*}}
\numberwithin{equation}{section}
 \theoremstyle{plain}
\newtheorem{theorem}{Theorem}[section]
\newtheorem{lemma}[theorem]{Lemma}

\theoremstyle{remark}
\newtheorem*{remark}{Remark}

\newcommand{\divides}{\mid}

\begin{document}

\title[On  Rogers--Ramanujan functions, binary quadratic forms and eta-quotients]{On  Rogers--Ramanujan functions, binary quadratic forms and eta-quotients}
\author{Alexander Berkovich}
\author{Hamza Yesilyurt}

\address{Department of Mathematics, University of Florida, 358 Little Hall,  Gainesville, FL
  32611, USA}\email{alexb@ufl.edu}
\address{Department of Mathematics, Bilkent University, 06800, Bilkent/Ankara, Turkey}
\email{hamza@fen.bilkent.edu.tr}

\thanks{Alexander Berkovich's research is partially supported by  grant H98230-09-1-0051 from the National Security Agency.}
\thanks{Hamza Yesilyurt's research is partially supported by  grant 109T669 from T\"ubitak.}

\begin{abstract}
In a handwritten
manuscript published with his lost notebook, Ramanujan stated
without proofs forty identities for the Rogers--Ramanujan functions.
We observe that the function that appears in Ramanujan's identities can be obtained from a Hecke action on a certain family  of eta products.
We  establish further Hecke-type relations for these functions involving binary quadratic forms.
Our observations enable us to find new identities for the Rogers--Ramanujan functions and also to use such identities in return to find identities involving binary quadratic forms.
\end{abstract}
\maketitle

\section{Introduction}

The Rogers--Ramanujan functions are defined for $ |q| < 1 $ by
\begin{equation*}
G(q) := \sum_{n=0}^{\i}\df{q^{n^2}}{(q;q)_n} \qquad \text{and} \qquad
H(q) := \sum_{n=0}^{\i}\df{q^{n(n+1)}}{(q;q)_n},
\end{equation*}
where $(a;q)_0 :=1$ and, for $n \geq 1$,
\begin{equation*}
(a;q)_n := \prod_{k=0}^{n-1}(1-aq^k).
\end{equation*}
These functions satisfy the famous Rogers--Ramanujan identities
\cite{rogers1894}
\begin{equation}\label{rridents}
G(q) = \df{1}{(q;q^5)_{\i}(q^4;q^5)_{\i}} \qquad \text{and} \qquad
H(q) = \df{1}{(q^2;q^5)_{\i}(q^3;q^5)_{\i}},
\end{equation}
where
\begin{equation*}
(a;q)_{\i} := \lim_{n\to\i}(a;q)_n, \qquad |q| <1.
\end{equation*}

In a handwritten manuscript published with his lost notebook \cite{lnb},
Ramanujan stated without proofs forty identities for the
Rogers--Ramanujan functions.
These identities were established in a
 series of papers by L.~J.~Rogers
 \cite{rogers} in 1921, G.~N.~Watson \cite{watson} in 1933,  D.~Bressoud
\cite{db} in 1977,    A.~J.~F.~Biagioli
\cite{tony} in 1989, and by the second author \cite{jmaa} in 2012. A detailed history of
Ramanujan's forty identities can be found in \cite{memoir}.

Ramanujan's identities mainly involve the function
\begin{equation*}
 U(r,s,q):=U(r,s)=\begin{cases}
         G(q^r)G(q^s)+q^{(s+r)/5}H(q^r)H(q^s)  & \text{if } s+r\equiv 0 \pmod 5,\\
         H(q^r)G(q^s)-q^{(s-r)/5}G(q^r)H(q^s)  & \text{if } s-r\equiv 0 \pmod 5.
\end{cases}
\end{equation*}

The modular properties of the function $U(r,s)$ were first established by Biagioli \cite{tony}. M.~ Koike \cite{koike} observed that for certain values of $r$ and $s$  the function $U(r,s)$  could be written  in terms of Thompson series. His observations were later proved by K. Bringmann and H. Swisher \cite{ppair} by using the theory of modular forms. However, it was not realized that this function occurs naturally as a Hecke action on eta products, as given by the following theorem.

\begin{theorem}\label{thr3}
Let $r$ and $s$ be two positive integers with $r+s \equiv 0 \pmod {24}$.

If $r+s \equiv 0 \pmod 5$, then
 \begin{equation}\label{thr3p1}
 T_5(\eta(r\tau)\eta(s\tau))=\eta(r\tau)\eta(s\tau)\left(q^{-(r+s)/60}U(r,s)\right)^2,
\end{equation}
and if $r-s \equiv 0 \pmod 5$, then
\begin{equation}\label{thr3p2}
T_5(\eta(r\tau)\eta(s\tau))=-\eta(r\tau)\eta(s\tau)\left(q^{(11r-s)/60}U(r,s)\right)^2,
\end{equation}
where in either case $\eta(\tau):=\sum_{n=-\infty}^\infty (-1)^n q^{(6n-1)^2/24}$ with $ q = \exp(2\pi i\tau) $ and  $ {\textrm{Im} } \, \tau > 0$.
\end{theorem}

Here and later in this manuscript, for any function to which we apply the Hecke action with prime 5 we have the reduction formula
\[T_5\left(\sum_{n=0}^{\i}a(n)q^n\right)=\sum_{n=0}^{\i}\left(a(5n)+a(n/5)\right)q^n,\]
where we employ the convention that $a(k) = 0$ whenever $k$ is not an integer. We will also make use of the notation
\[U_5\left(\sum_{n=0}^{\i}a(n)q^n\right) := \sum_{n=0}^{\i}a(5n)q^n.\]

During our investigations, we also observed the following two theorems, where $\t(q) := \sum_{n=-\i}^{\i}q^{n^2}$.

\begin{theorem}\label{thr1}
If $r+s\equiv 0 \pmod 5$, then the following two identities hold:
\begin{align}
2\t(q^r)\t(q^s)+T_5(\t(q^r)\t(q^s)) &= 4E(q^{2r})E(q^{2s})U(r,s)U(4r,4s)\label{he1}\\
2\t(q^r)\t(q^s)-T_5(\t(q^r)\t(q^s)) &= 4q^rE(q^{2r})E(q^{2s})U(r,4s)U(4r,s).\label{he2}
\end{align}
\end{theorem}

\begin{theorem}\label{thr2}
If $r-s\equiv 0 \pmod 5$, then the following two identities hold:
\begin{align*}
2\t(q^r)\t(q^s)+T_5(\t(q^r)\t(q^s)) &= 4E(q^{2r})E(q^{2s})U(r,4s)U(s,4r),\\
2\t(q^r)\t(q^s)-T_5(\t(q^r)\t(q^s)) &= 4q^rE(q^{2r})E(q^{2s})U(r,s)U(4r,4s).
\end{align*}
\end{theorem}

Let $(a,b,c)$ denote the positive definite quadratic form $an^2+bnm+cn^2$ with discriminant $D=b^2-4ac<0$. For simplicity, we will not distinguish the quadratic form $(a,b,c)$ and its  generating function  $\sum_{n,m=-\i}^{\i}q^{an^2+bnm+cm^2}$.
It is well known that a Hecke action on a binary quadratic form of a given discriminant can be written as a linear combination of the quadratic forms of this discriminant \cite[p.~794]{Hecke}. In applications of Theorem \ref{thr3}, we first express the eta product as a linear combination of the relevant quadratic forms. In this format, Theorems \ref{thr1}--\ref{thr2} are easier to apply. These theorems enable us to find new identities for the Rogers--Ramanujan functions and also to use such identities in return to find identities involving binary quadratic forms.  Among the many such results presented in this paper we give two examples (see \eqref{tkmm4} and \eqref{ap5}): letting $\k(q):=(-q;q^2)_\i$, they are

\begin{equation*}
2qU(1,71,q^2)=-2q^3+\k(q)\k(q^{71})-\k(-q)\k(-q^{71})-2q^9\df{1}{\k(-q^2)\k(-q^{142})}
\end{equation*}
and
\begin{equation*}
\df{(1,1,10)+(2,1,5)-(1,0,39)-(5,2,8)}{(3,0,13)+(2,1,5)-(3,3,4)-(5,2,8)}=\df{\t(-q^6)\t(-q^{26})}{\t(-q^2)\t(-q^{78})}.
\end{equation*}

We proceed by collecting the necessary definitions and formulas in the next section. In Section \ref{tHr}, we give proofs of Theorems \ref{thr3}--\ref{thr2}. In Section \ref{appl}, we present several applications.  We conclude in the last section with a brief description of the prospects for future work.

\section{Definitions and Preliminary Results}

We first recall  Ramanujan's definitions for a general theta
function and some of its important special cases.  Set
\begin{equation}\label{generaltheta}
f(a,b) := \sum_{n=-\i}^{\i}a^{n(n+1)/2}b^{n(n-1)/2}, \qquad |ab| <
1.
\end{equation}
The function $ f(a,b) $ satisfies the
well-known Jacobi triple product identity \cite[p.~35, Entry
19]{III}
\begin{equation}\label{19III}
f(a,b) = (-a;ab)_{\i}(-b;ab)_{\i}(ab;ab)_{\i}.
\end{equation}
The three most important special cases of \eqref{generaltheta} are
\begin{equation}\label{22i}
\varphi(q) := f(q,q) = \sum_{n=-\i}^{\i}q^{n^2} =
(-q;q^2)_{\i}^2(q^2;q^2)_{\i},
\end{equation}
\begin{equation*}
\psi(q) := f(q,q^3) = \sum_{n=0}^{\i}q^{n(n+1)/2} =
\df{(q^2;q^2)_{\i}}{(q;q^2)_{\i}},
\end{equation*}
and
\begin{equation}\label{22iii}
E(q) := f(-q,-q^2) = \sum_{n=-\infty}^\infty (-1)^n q^{n(3n-1)/2}
= (q;q)_\infty = q^{-1/24}\eta(\tau).
\end{equation}
%where $ q = \exp(2\pi i\tau), $  $ \text{Im } \tau > 0,$ and  $
%\eta $ denotes the Dedekind eta-function.
The product
representations in \eqref{22i}--\eqref{22iii} are special cases of
\eqref{19III}.
The function $ f(a,b) $ also
satisfies a useful addition formula.  For each  integer $
n$, let
\begin{equation*}
U_n := a^{n(n+1)/2}b^{n(n-1)/2} \qquad \text{and} \qquad V_n :=
a^{n(n-1)/2}b^{n(n+1)/2}.
\end{equation*}
Then \cite[p.~48, Entry 31]{III}
\begin{equation}\label{31III}
f(U_1,V_1) =
\sum_{r=0}^{n-1}U_rf\left(\df{U_{n+r}}{U_r},\df{V_{n-r}}{U_r}\right).
\end{equation}
With $a=b=q$ and $n=2$, we find from \eqref{31III} that
\begin{equation}\label{tdis}
\t(q)=\t(q^4)+2q\psi(q^8).
\end{equation}
Similarly, with $a=q$, $b=q^3$, and $n=2$, we find that
\begin{equation}\label{pdis}
\psi(q)=f(q^6,q^{10})+qf(q^2,q^{14}).
\end{equation}
By \eqref{rridents} and \eqref{19III}, we see that
\begin{equation}\label{ghd}
G(q)=\df{f(-q^2,-q^3)}{E(q)} \quad \text{and} \quad H(q)=\df{f(-q,-q^4)}{E(q)}.
\end{equation}
A useful consequence of \eqref{ghd} in conjunction
with the Jacobi triple product identity \eqref{19III} is
\begin{equation}\label{ghp}
G(q)H(q) = \df{E(q^5)}{E(q)}.
\end{equation}

The odd-even dissections of $G$ and $H$ were given by Watson
\cite{watson}:

\begin{equation}\label{gh2dis}
\begin{split}
G(q)&=\df{E(q^8)}{E(q^2)}\bigr(G(q^{16})+qH(-q^4)\bigl),\\
H(q)&=\df{E(q^8)}{E(q^2)}\bigr(q^3H(q^{16})+G(-q^4)\bigl).
\end{split}
\end{equation}

Recall that  the general theta
function $f$ is defined by \eqref{generaltheta}. For convenience, we  also define

\begin{equation*}
 f_k(a,b) := \begin{cases}
         f(a,b)   & \text{if } k \equiv 0 \pmod 2,\\
         f(-a,-b) & \text{if } k \equiv 1 \pmod 2.
 \end{cases}
\end{equation*}

Let $m$ be an integer and let $\a$, $\b$, $p$, and $\l$ be positive
integers such that
\begin{equation*}
\a m^2+\b=p\l.
\end{equation*}
Let $\d$ and $\e$ be integers, and let $t$ and $l$ be reals. With the parameters defined this way, we set
\begin{equation}\label{RGdef}
\begin{split}
R(\e,\d,l,t,\a,\b,m,p,\l) :=
  \sum_{\substack{k=0\\n:=2k+t}}^{p-1}\Bigl(&(-1)^{\e k}q^{\{\l n^2+p\a l^2+2\a nml\}/8}\\
                                           &\times f_\d(q^{(1+l)p\a/2+\a nm/2},q^{(1-l)p\a/2-\a nm/2})\\
                                           &\times f_{\e p/2+m\d/2}(q^{p\b/2+\b n/2},q^{p\b/2 - \b n/2})\Bigr).
\end{split}
\end{equation}
Then, by \cite[Th.~1]{mjnt} with $x=y=1$ and $q$ replaced by $q^{1/2}$, we have
\begin{equation}\label{st3}
 R(\e,\d,l,t,\a,\b,m,p,\l)=\sum_{u,v=-\i}^{\i}(-1)^{\d v+ \e
 u}q^{(\l U^2 +2\a m UV +p\a V^2)/8},
\end{equation}
where $U:=2u+t$ and  $V:=2v+l$.
From this representation it follows that \cite[Cor.~3.2]{mjnt}
\begin{equation}\label{prtt1}
R(\e,\d,l,t,\a,\b,m,p,\l)=R(\d,\e,t,l,1,\a \b, \a m, \l,p\a).
\end{equation}
Moreover, we have the following lemma.
\begin{lemma}\cite[Cor.~3.3]{mjnt}
Let $\a_1$, $\b_1$, $m_1$, $p_1$ be another set of parameters such
that $\a_1m_1^2+\b_1=p_1\l$, $\a\b=\a_1\b_1$, and $\l \divides (\a
m-\a_1 m_1)$. Set
\begin{equation*}
 a:=\df{\a m-\a_1 m_1}{\l}.
\end{equation*}
Then,
\begin{equation}\label{prt1}
R(\e,\d,l,t,\a,\b,m,p,\l)=R(\e,\d+a\e,l,t+al,\a_1,\b_1,m_1,p_1,\l).
\end{equation}
\end{lemma}
From \eqref{st3} we then conclude that
\begin{equation}\label{qrr}
(a,b,c)=R(0,0,0,0,1,-D,b,2c,2a).
\end{equation}

\section{Hecke-Type Relations}\label{tHr}
In this section we presents proofs of Theorems \ref{thr3}--\ref{thr2}.

\subsection*{Proof of Theorem \ref{thr3}}
 The proofs of \eqref{thr3p1} and \eqref{thr3p2} are essentially same, so we will only prove \eqref{thr3p1}. For simplicity we set $Q:=q^5$.
We start with the well-known 5-dissection of $E(q)$ as found in \cite[p.~81, Entry~38(iv)]{III}:
\begin{equation}\label{b1}
E(q)E(Q)=f^2(-Q^2,-Q^3)-q^2f^2(-Q,-Q^4)-qE(Q)E(Q^5).
\end{equation}
Using \eqref{ghd}, we can write \eqref{b1} in its equivalent form
\begin{equation}\label{b2}
\df{E(q)}{E(Q)}=G^2(Q)-q^2H^2(Q)-q\df{E(Q^5)}{E(Q)}.
\end{equation}
From \eqref{b2}, we find that
\begin{equation*}
\df{E(q^r)E(q^s)}{E(Q^{r})E(Q^{s})}=\left(G^2(Q^r)-q^{2r}H^2(Q^r)-q^r\df{E(Q^{5r})}{E(Q^r)}\right)\left(G^2(Q^s)-q^{2s}H^2(Q^s)-q^s\df{E(Q^{5s})}{E(Q^s)}\right),
\end{equation*}
from which we deduce that
\begin{align*}
\df{U_5(E(q^r)E(q^s))}{E(q^r)E(q^s)}
&=G^2(q^r)G^2(q^s)+q^{2(r+s)/5}H^2(q^r)H^2(q^s)+q^{(r+s)/5}\df{E(Q^{r})E(Q^s)}{E(q^r)E(q^s)}\\
&=U(r,s)^2-2q^{(r+s)/5}G(q^r)G(q^s)H(q^r)H(q^s)+q^{(r+s)/5}\df{E(Q^{r})E(Q^s)}{E(q^r)E(q^s)}\\
&=U(r,s)^2-q^{(r+s)/5}\df{E(Q^{r})E(Q^s)}{E(q^r)E(q^s)},
\end{align*}
where in the last step we used \eqref{ghp}. Therefore, we have that
\begin{equation}\label{b5}
U_5(E(q^r)E(q^s))+q^{(r+s)/5}E(Q^{r})E(Q^s)=E(q^r)E(q^s)U(r,s)^2.
\end{equation}
Finally, we multiply both sides of \eqref{b5} by $q^{(r+s)/120}$ and use the fact that $\eta(\tau)=q^{1/24}E(q)$ to arrive at
\begin{equation*}
U_5(\eta(r\tau)\eta(s\tau))+\eta(5r\tau)\eta(5s\tau)=\eta(r\tau)\eta(s\tau)\left(q^{-(r+s)/60}U(r,s)\right)^2,
\end{equation*}
which is clearly equivalent to \eqref{thr3p1}.
\qed

\subsection*{Proof of Theorems \ref{thr1} and \ref{thr2}}
The proofs of the Theorems \ref{thr1} and \ref{thr2} are identical, so we will only prove Theorem \ref{thr1}. For convenience, we again set $Q:=q^5$.
By \eqref{31III}, with $a=b=q$ and $n=5$, and by \eqref{22i}, we get
\begin{equation}\label{a1}
\t(q)=\t(Q^5)+2qf(Q^3,Q^7)+2q^4f(Q,Q^9).
\end{equation}
From \eqref{19III}, with simple product manipulations we find that
\begin{equation}
A(q):=f(q^3,q^7)=E(q^2)H(q)G(q^4) \quad \text{and} \quad B(q):=f(q,q^9)=E(q^2)G(q)H(q^4).\label{a2}
\end{equation}
We have from \eqref{a1} and \eqref{a2} that
\begin{equation}
\t(q^r)\t(q^s)=\left(\t(Q^{5r})+2qA(Q^r)+2q^4B(Q^r)\right)\left(\t(Q^{5s})+2qA(Q^s)+2q^4B(Q^s)\right).\label{a3}
\end{equation}
From \eqref{a3}, together with the fact that $r+s\equiv 0 \pmod 5$, we conclude that
\begin{equation}
U_5\left(\t(q^r)\t(q^s)\right)=\t(Q^r)\t(Q^s)+4q^{(r+s)/5}A(q^r)A(q^s)+4q^{4(r+s)/5}B(q^r)B(q^s).\label{a4}
\end{equation}

The following two identities of Ramanujan \cite[Entries 2 and 3]{memoir} will be employed in our proofs:
\begin{equation}\label{2}
G(q)G(q^4)+qH(q)H(q^4) = \df{\varphi(q)}{E(q^2)}.
\end{equation}
and
\begin{equation}\label{3}
G(q)G(q^4)-qH(q)H(q^4)  = \df{\varphi(q^5)}{E(q^2)}.
\end{equation}

  From \eqref{a2}, we have
  \begin{equation}\label{a5}
    \begin{split}
	 4q^{(r+s)/5}A(q^r)A(q^s)+4&q^{4(r+s)/5}B(q^r)B(q^s)\\
	 &= 4E(q^{2r})E(q^{2s})\left(q^{(r+s)/5}H(q^r)G(q^{4r})H(q^s)G(q^{4s})\right.\\
	 &\quad \left.+q^{4(r+s)/5}G(q^r)H(q^{4r})G(q^s)H(q^{4s})\right)\\
	 &= 4E(q^{2r})E(q^{2s})\left(G(q^r)G(q^s)+q^{(r+s)/5}H(q^r)H(q^s)\right)\\
	 &\quad \times\left(G(q^{4r})G(q^{4s})+q^{4(r+s)/5}H(q^{4r})H(q^{4s})\right) - 4E(q^{2r})E(q^{2s})\\
	 &\quad \times\left(G(q^r)G(q^s)G(q^{4r})G(q^{4s})+q^{r+s}H(q^r)H(q^s)H(q^{4r})H(q^{4s})\right)\\
	 &= 4E(q^{2r})E(q^{2s})U(r,s)U(4r,4s) - \left(\t(q^r)+\t(Q^r)\right)\left(\t(q^s)+\t(Q^s)\right)\\
	 &\quad -\left(\t(q^r)-\t(Q^r)\right)\left(\t(q^s)-\t(Q^s)\right)\\
	 &= 4E(q^{2r})E(q^{2s})U(r,s)U(4r,4s)-2\left(\t(q^r)\t(q^s)+\t(Q^r)\t(Q^s)\right),
    \end{split}
  \end{equation}
  where in the next to last step we use \eqref{2} and \eqref{3}. We now return to \eqref{a4} and use \eqref{a5} to find that
\begin{equation*}
2\t(q^r)\t(q^s)+U_5\left(\t(q^r)\t(q^s)\right)+\t(Q^r)\t(Q^s)=4E(q^{2r})E(q^{2s})U(r,s)U(4r,4s),
\end{equation*}
which is clearly equivalent to \eqref{he1}.

While we can prove \eqref{he2} exactly the same way we proved \eqref{he1} by simply grouping terms differently in \eqref{a5}, we can also give a direct proof by showing that
\begin{equation}\label{a7}
\t(q^r)\t(q^s)=E(q^{2r})E(q^{2s})\left(U(r,s)U(4r,4s)+q^rU(r,4s)U(4r,s)\right).
\end{equation}
To prove \eqref{a7}, we consider the system of equations
\begin{align*}
U(r,s) &= G(q^r)G(q^{s})+q^{(r+s)/5}H(q^r)H(q^{s}),\\
U(r,4s) &= -q^{(4s-r)/5}G(q^r)H(q^{4s})+H(q^r)G(q^{4s}),\\
\df{\t(q^r)}{E(q^{2r})} &= G(q^{r})G(q^{4r})+q^rH(q^r)H(q^{4r}),
\end{align*}
where the last equation is simply \eqref{2}.
It follows that
\begin{equation*}
\left|\begin{matrix}
U(r,s)                  & G(q^{s})               & q^{(r+s)/5} H(q^s) \\
U(r,4s)                 & -q^{(4s-r)/5}H(q^{4s}) & G(q^{4s}) \\
\df{\t(q^r)}{E(q^{2r})} & G(q^{4r})              & q^rH(q^{4r})
\end{matrix}\right|=0.
\end{equation*}
By expanding this determinant we discover that
  \begin{align*}
    0 &= U(r,s)\left(-q^{(r+s)/5}H(q^{4s})H(q^{4r})-G(q^{4r})G(q^{4s})\right)\\
	 &\quad -U(r,4s)\left(q^rG(q^s)H(q^{4r})-q^{(r+s)/5}H(q^s)H(q^{4r})\right)\\
	 &\quad +\dfrac{\t(q^r)}{E(q^{2r})}\left(G(q^s)G(q^{4s})+q^sH(q^s)H(q^{4s})\right)\\
	 &= -U(r,s)U(4r,4s)-q^rU(r,4s)U(4r,s)+\dfrac{\t(q^r)\t(q^s)}{E(q^{2r})E(q^{2s})},
  \end{align*}
which is \eqref{he2}.
\qed

\section{Applications}\label{appl}
The first set of identities we will prove involves the quadratic forms $(1,1,10)$, $(2,1,5)$, and $(3,3,4)$ of discriminant $-39$ and the quadratic forms $(1,0,39)$, $(3,0,13)$, and $(5,2,8)$ of discriminant $-156$.
From \eqref{prt1}, we observe that $R(0,0,0,0,1,39,1,4,10)=R(0,0,0,0,3,13,-3,4,10)$.  By \eqref{qrr}, we also have $(2,1,5)=(5,1,2)=R(0,0,0,0,1,39,1,4,10)$.
For simplicity we now set $Q:=q^{13}$. From \eqref{RGdef}, we find that
\begin{align*}
R(0,0,0,0,1,39,1,4,10)&=f(q^2,q^2)f(Q^6,Q^6)+2q^5f(q,q^3)f(Q^3,Q^9)+q^{20}f(1,q^4)f(1,Q^{12})\\
&=\t(q^2)\t(Q^6)+2q^5\psi(q)\psi(Q^3)+4q^{20}\psi(q^4)\psi(Q^{12}).
\end{align*}
Similarly,
\begin{align*}
R(0,0,0,0,3,13,-3,4,10)&=f(q^6,q^6)f(Q^2,Q^2)+2q^2f(q^3,q^9)f(Q,Q^3)+q^8f(1,q^{12})f(1,Q^4)\\
&=\t(q^6)\t(Q^2)+2q^2\psi(q^3)\psi(Q)+4q^{8}\psi(q^{12})\psi(Q^{4}).
\end{align*}
Hence, we conclude that
\begin{equation}\label{215}
\begin{split}
(2,1,5) &= \t(q^2)\t(Q^6)+2q^5\psi(q)\psi(Q^3)+4q^{20}\psi(q^4)\psi(Q^{12})\\
&= \t(q^6)\t(Q^2)+2q^2\psi(q^3)\psi(Q)+4q^{8}\psi(q^{12})\psi(Q^{4}).
\end{split}
\end{equation}
By similar considerations one can also obtain
\begin{equation}\label{1110}
(1,1,10)=\t(q)\t(Q^3)+4q^{10}\psi(q^2)\psi(Q^6),
\end{equation}
\begin{equation*}
(3,3,4)=\t(q^3)\t(Q)+4q^4\psi(q^6)\psi(Q^2),
\end{equation*}
and
\begin{equation}\label{528}
\begin{split}
(5,2,8) &= \t(q^{24})\t(Q^8)+2{q}^{8}\psi(q^{12})\psi(Q^4) +4q^{32}\psi(q^{48})\psi(Q^{16})\\
 &\quad +2{q}^{5}f({q}^{6},{q}^{42}) f({Q}^{6},{Q}^{10}) +2{q}^{15}f({q}^{18},{q}^{30}) f({Q}^{2},{Q}^{14}) \\
&= \t(q^{8})\t(Q^{24})+2{q}^{20}\psi(q^{4})\psi(Q^{12}) +4q^{80}\psi(q^{16})\psi(Q^{48})\\
 &\quad +2{q}^{5}f({Q}^{18},{Q}^{30}) f({q}^{6},{q}^{10}) +2{q}^{45}f({Q}^{6},{Q}^{42}) f({q}^{2},{q}^{14}).
\end{split}
\end{equation}

\begin{theorem} The following four facts are true (with $Q:=q^{13}$):
\begin{equation}\label{ap2}
  2q^2E(q^2)E(Q^6)U(1,39)U(1,39,-q) = (1,1,10)+(2,1,5)-(1,0,39)-(5,2,8),
\end{equation}
\begin{equation}\label{ap3}
  2q^2E(q^6)E(Q^2)U(3,13)U(3,13,-q) = (3,0,13)+(2,1,5)-(3,3,4)-(5,2,8),
\end{equation}
\begin{equation}\label{ap4}
  \df{(3,0,13)-(5,2,8)}{(1,0,39)+(5,2,8)} = q^3\df{\psi(-q)\psi(-Q^3)}{\psi(-q^3)\psi(-Q)},
\end{equation}
\begin{equation}\label{ap5}
  \df{(1,1,10)+(2,1,5)-(1,0,39)-(5,2,8)}{(3,0,13)+(2,1,5)-(3,3,4)-(5,2,8)} = \df{\t(-q^6)\t(-Q^2)}{\t(-q^2)\t(-Q^6)}.
\end{equation}
\end{theorem}
 \begin{proof}
 We start by proving \eqref{ap2}. Using Theorem \ref{thr1}, we see that
 \begin{equation}\label{fpr1}
 4E(q^2)E(Q^6)U(1,39)U(1,39,q^4)=2(1,0,39)+T_5(1,0,39)=2(1,0,39)+2(5,2,8).
 \end{equation}
 From the odd-even dissections of the functions $G(q)$ and $H(q)$, i.e.\ \eqref{gh2dis}, we have
 \begin{equation*}
 \df{E(q^2)E(Q^6)}{E(q^8)E(Q^{24})}U(1,39)=U(1,39,q^{16})+q^{8}U(1,39,-q^4)+qU(1,156,-q^4)+q^{39}U(39,4,-q^4).
 \end{equation*}
From \eqref{tdis} we get
\begin{equation}\label{trsdis}
\t(q^r)\t(q^s)=\t(q^{4r})\t(q^{4s})+4q^{r+s}\psi(q^{8r})\psi(q^{8s})+2q^s\t(q^{4r})\psi(q^{8s})+2q^r\t(q^{4s})\psi(q^{8r}).
\end{equation}
Examining \eqref{1110} and \eqref{trsdis}, we observe that the even part of $(1,0,39)$ is $(1,1,10,q^4)$.
From \eqref{215} and \eqref{528}, we immediately see that the even part of $(5,2,8)$ is $(1,1,10,q^4)$. Therefore, by equating the even parts in both sides of \eqref{fpr1}, we conclude that
\begin{equation*}
2E(q^8)E(Q^{24})\left(U(1,39,q^{16})+q^{8}U(1,39,-q^4)\right)U(1,39,q^4)=(1,1,10,q^4)+(2,1,5,q^4).
\end{equation*}
Next, we replace $q^4$ by $q$ and use \eqref{fpr1} to arrive at \eqref{ap2}. To prove \eqref{ap3}, we use Theorem \ref{thr2} and  start with the identity
\begin{equation}\label{fpr7}
4q^3E(q^6)E(Q^2)U(3,13)U(3,13,q^4)=2(3,0,13)-T_5(3,0,13)=2(3,0,13)-2(5,2,8).
\end{equation}
By looking at the even parts in both sides of this equation and arguing as before, we arrive at \eqref{ap2}.
Ramanujan observed that \cite[Entry 3.19]{memoir}
\begin{equation}\label{rmp}
E(q)E(Q^3)U(1,39)=E(q^3)E(Q)U(3,13).
\end{equation}
By \eqref{ap2}, \eqref{ap3}, and \eqref{rmp}, and using some elementary product manipulations, we deduce \eqref{ap5}. Similarly, \eqref{ap4} follows from \eqref{fpr1}, \eqref{fpr7}, and \eqref{rmp}.
\end{proof}

\begin{remark}
It can be easily verified by appealing to theory of modular forms that
\begin{equation*}
2E(q^3)E(Q)U(1,39)=(1,1,10)+(2,1,5)
\end{equation*}
and
\begin{equation*}
2q^2E(q)E(Q^3)U(3,13)=(2,1,5)-(3,3,4).
\end{equation*}
It is also easy to prove (see, for example, the proof of \eqref{tkmm3}) that
\begin{equation*}
2qE(q)E(Q^3)U(1,39)=2qE(q^3)E(Q)U(3,13)= (1,1,10)-(3,3,4).
\end{equation*}
From these last three equations we easily observe that
\begin{equation*}
\sqrt{\df{ (2,1,5)-(3,3,4)}{(1,1,10)+(2,1,5)}}=\df{(2,1,5)-(3,3,4)}{(1,1,10)-(3,3,4)}=\df{(1,1,10)-(3,3,4)}{(1,1,10)+(2,1,5)}=q\df{E(q)E(Q^3)}{E(q^3)E(Q)}.
\end{equation*}
\end{remark}

Next, we treat identities involving the quadratic forms $(2,1,44)$, $(8,1,11)$, $(4,1,22)$, $(10,7,10)$, $(5,3,18)$, $(1,1,88)$, and $(9,3,10)$ of discriminant $-351$.

\begin{theorem} The following five facts are true (with $Q := q^{13}$):
\begin{equation}\label{tk1}
(2,1,44)-(8,1,11)=2q^2E(q^9)E(Q^3),
\end{equation}
\begin{equation}\label{tk2}
(5,3,18)-(8,1,11)=2q^5E(q^3)E(Q^9),
\end{equation}
\begin{equation}\label{tk3}
(4,1,22)-(10,7,10)=2q^4E(q^9)E(Q^3)U(3,13,q^3)^2,
\end{equation}
\begin{equation}\label{tk4}
(1,1,88)-(9,3,10)=2qE(q^3)E(Q^9)U(1,39,q^3)^2,
\end{equation}
\begin{equation}\label{tk5}
\df{(4,1,22)-(10,7,10)}{(1,1,88)-(9,3,10)}=\df{(5,3,18)-(8,1,11)}{(2,1,44)-(8,1,11)}=q^3\df{E(q^3)E(Q^9)}{E(q^9)E(Q^3)}.
\end{equation}
\end{theorem}
\begin{proof}
Using \eqref{qrr} and \eqref{RGdef}, we get
\begin{equation}\label{tk6}
\begin{split}
(2,1,44) &= (44,1,2) = R(0,0,0,0,1,351,1,4,88)\\
         &= \t(q^2)\t(Q^{54})+2q^{44}\psi(q)\psi(Q^{27})+4q^{176}\psi(q^4)\psi(Q^{108})
\end{split}
\end{equation}
as well as
\begin{equation}\label{tk7}
\begin{split}
(8,1,11) &= (11,1,8) = R(0,0,0,0,1,351,1,16,22)\\
         &= \t(q^8)\t(Q^{216})+2q^{11}f(q^7,q^{9})f(Q^{189},Q^{243})+2q^{44}f(q^6,q^{10})f(Q^{162},Q^{270})\\
         &\quad +2q^{99}f(q^5,q^{11})f(Q^{135},Q^{297})+2q^{176}\psi(q^4)\psi(Q^{108})+2q^{275}f(q^3,q^{13})f(Q^{81},Q^{351})\\
         &\quad +2q^{396}f(q^2,q^{14})f(Q^{54},Q^{378})+2q^{539}f(q,q^{15})f(q^{27},q^{405})+4q^{704}\psi(q^{16})\psi(Q^{432}).
\end{split}
\end{equation}
We take \eqref{tk6} and we expand $\t(q^2)\t(Q^{54})$ by using \eqref{tdis} and we similarly expand $\psi(q)\psi(Q^{27})$ by using \eqref{pdis}. After this expansion we subtract \eqref{tk7} from the expanded \eqref{tk6} and arrive at
\begin{equation}\label{tk11}
\begin{split}
(2,1,44)-(8,1,11) &= 2{q}^{2}\psi({q}^{16})\t( {Q}^{216}) -2{q}^{11}f({q}^{7},{q}^{9}) f({Q}^{189},{Q}^{243})\\
                  &\quad +2{q}^{45}f({q}^{2},{q}^{14}) f({Q}^{162},{Q}^{270}) -2{q}^{99}f({q}^{5},{q}^{11}) f({Q}^{135},{Q}^{297})\\
                  &\quad +2{q}^{176}f({q}^{4},{q}^{12}) f({Q}^{108},{Q}^{324})-2{q}^{275}f({q}^{3},{q}^{13}) f({Q}^{81},{Q}^{351})\\
                  &\quad +2{q}^{395}f({q}^{6},{q}^{10}) f({Q}^{54},{Q}^{378}) -2{q}^{539}f(q,{q}^{15}) f({Q}^{27},{Q}^{405})\\
                  &\quad +2{q}^{702}\t({q}^{8})\psi({Q}^{432}).
\end{split}
\end{equation}
By \eqref{RGdef}, we have
\begin{equation*}
R(0,1,0,1,9,39,1,3,16)=2q^2E(q^9)E(Q^3).
\end{equation*}
Next, we employ \eqref{prtt1} and find that
$R(0,1,0,1,9,39,1,3,16)=R(1,0,1,0,1,351,9,16,27)$. By employing \eqref{RGdef} one more time we observe that $R(1,0,1,0,1,351,9,16,27)$ equals exactly the right side of \eqref{tk11}, which completes the proof of \eqref{tk1}.

We now observe that
\begin{equation*}
2T_5(2,1,44)=(9,3,10)+(10,7,10) \quad \text{and} \quad 2T_5(8,1,11)=(4,1,22)+(9,3,10).
\end{equation*}
Therefore, by \eqref{tk1}, we find that
\begin{equation*}
T_5(2q^2E(q^9)E(Q^3))=(10,7,10)-(4,1,22),
\end{equation*}
from which by way of \eqref{thr3p2} we immediately arrive at \eqref{tk3}.

The proofs of \eqref{tk2} and \eqref{tk4} go along the same lines as the proofs of \eqref{tk1} and \eqref{tk3}, respectively,  so we omit them. In fact one can go from one identity to the other via the map $\tau \mapsto -351/\tau$.

Finally, \eqref{tk5} follows from \eqref{rmp}.
\end{proof}

\begin{remark}
By a straightforward but quite lengthy argument one can eliminate $q^3$ from either \eqref{tk2} or \eqref{tk4}, resulting in
\begin{equation}\label{bse1}
\begin{split}
\left(E(q^3)E(Q)U(3,13)\right)^2 &= \left(E(q)E(Q^3)U(1,39)\right)^2\\
                                 &=\t(-q^2)\t(-Q^6)\psi(-Q)\psi(-q^3)-q^3\t(-Q^2)\t(-q^6)\psi(-q)\psi(-Q^3).
\end{split}
\end{equation}
From \eqref{bse1} we may then deduce the following identity which is similar to those found by Ramanujan:
\begin{equation*}
U(1,39)U(3,13)=\df{\k(q)\k(Q^3)}{\k(-q^6)\k(-Q^2)}-q^3\df{\k(Q)\k(q^3)}{\k(-Q^6)\k(-q^2)}.
\end{equation*}
\end{remark}

The next theorem concerns relations for the quadratic forms $(1,1,18)$, $(2,1,9)$, $(4,3,5)$, and $(3,1,6)$ of discriminant $-71$. Here we now set $Q:=q^{71}$.

\begin{theorem} The following five facts are true (with $Q := q^{71}$):
\begin{equation}\label{tkmm1}
2q^3E(q)E(Q)=(3,1,6)-(4,3,5),
\end{equation}
\begin{equation}\label{tkmm2}
2qE(q)E(Q)U(1,71)^2=(3,1,6)-(4,3,5)-(2,1,9)+(1,1,18),
\end{equation}
\begin{equation}\label{tkmm3}
2q^2E(q)E(Q)U(1,71)=(2,1,9)-(3,1,6),
\end{equation}
\begin{equation}\label{tkmm4}
2qU(1,71,q^2)=-2q^3+\k(q)\k(Q)-\k(-q)\k(-Q)-2q^9\df{1}{\k(-q^2)\k(-Q^2)},
\end{equation}
\begin{equation}\label{tkmm5}
\left((3,1,6)-(4,3,5)-(2,1,9)+(1,1,18)\right)\left((3,1,6)-(4,3,5)\right)=\left((2,1,9)-(3,1,6)\right)^2.
\end{equation}
\end{theorem}
\begin{proof}
The proof of \eqref{tkmm1} is similar to that of \eqref{tk1} so we omit the details.

The identity \eqref{tkmm2} follows from \eqref{thr3p2} once we observe that
\begin{equation*}
T_5(3,1,6)=(4,3,5)+(2,1,9) \quad \text{and} \quad T_5(4,3,5)=(3,1,6)+(1,1,18).
\end{equation*}

Using \eqref{RGdef} and \eqref{ghd}, we get
\begin{equation}\label{tkmm6}
\begin{split}
R(0,1,0,1,1,71,3,5,16) &= 2{q}^{2}f(-q,-{q}^{4}) f(-{Q}^{2},-{Q}^{3}) -2{q}^{16}f(-{q}^{2},-{q}^{3}) f(-Q,-{Q}^{4})\\
                       &= 2q^2E(q)E(Q)U(1,71).
\end{split}
\end{equation}
From \eqref{prtt1} and \eqref{RGdef}, we also find that
\begin{equation}\label{oiu}
\begin{split}
R(0,1,0,1,1&,71,3,5,16)\\
&= R(1,0,1,0,1,71,3,16,5)\\
&= 2{q}^{2}\psi({q}^{16})\t(Q^8) -2{q}^{3}f({q}^{3},{q}^{13}) f({Q}^{7},{Q}^{9}) +2{q}^{9}f({q}^{6},{q}^{10}) f({Q}^{6},{Q}^{10})\\
&\quad -2{q}^{20}f({q}^{7},{q}^{9}) f({Q}^{5},{Q}^{11}) +2{q}^{36}\psi({q}^{4}) \psi({Q}^{4}) -2{q}^{57}f(q,{q}^{15}) f({Q}^{3},{Q}^{13})\\
&\quad +2{q}^{81}f({q}^{2},{q}^{14}) f({Q}^{2},{Q}^{14}) -2{q}^{109} f({q}^{5},{q}^{11}) f(Q,{Q}^{15}) +2{q}^{142}\t({q}^{8})\psi({Q}^{16}).
\end{split}
\end{equation}
By \eqref{qrr} and \eqref{RGdef}, we have
\begin{equation}\label{atk5}
\begin{split}
(2,1,9)=(9,2,1)&=R(0,0,0,0,1,71,1,4,18)\\
&=\t(q^2)\t(Q^2)+2q^9\psi(q)\psi(Q)+4q^{36}\psi(q^4)\psi(Q^4).
\end{split}
\end{equation}
From \eqref{qrr}, \eqref{prtt1}, and \eqref{RGdef}, we deduce that
\begin{equation}\label{atk7}
\begin{split}
(3,1,6)
&= R(0,0,0,0,1,71,1,12,6)\\
&= R(0,0,0,0,1,71,-5,16,6)\\
&= \t(q^8)\t(Q^8)+2q^3f(q^3,q^{13})f(Q^7,Q^9)+2q^{10}f(q^2,q^{14})f(Q^6,Q^{10})\\
&\quad +2q^{20}f(q^7,q^9)f(Q^5,Q^{11})+2q^{36}\psi(q^4)\psi(Q^4)+2q^{57}f(q,q^{15})f(Q^3,Q^{13})\\
&\quad +2q^{80}f(q^6,q^{10})f(Q^2,Q^{14})+2q^{109}f(q^5,q^{11})f(Q,Q^{15})+4q^{144}\psi(q^{16})\psi(Q^{16}).
\end{split}
\end{equation}
We then take \eqref{atk5}, expand $\t(q^2)\t(Q^{54})$ as per \eqref{tdis} and expand $\psi(q)\psi(Q^{27})$ as per \eqref{pdis}, and then subtract \eqref{atk7} to obtain the right-hand side of \eqref{oiu}. Then by \eqref{oiu} and \eqref{tkmm6}, the proof of \eqref{tkmm3} is complete.

Next observe that by \eqref{qrr} and \eqref{RGdef} we have
\begin{align*}
(4,3,5)
&= (5,4,3)\\
&= R(0,0,0,0,1,71,3,8,10)\\
&= \t(q^4)\t(Q^4)+2q^5f(q,q^7)f(Q^3,Q^5)+2q^{18}\psi(q^2)\psi(Q^2)\\
&\quad +2q^{40}f(q^3,q^5)f(Q,Q^7)+4q^{72}\psi(q^8)\psi(Q^8)\\
&= \left(\t(q)\t(Q)+\t(-q)\t(-Q)\right)/2+2q^{18}\psi(q^2)\psi(Q^2)\\
&\quad +2q^5f(q,q^7)f(Q^3,Q^5)+2q^{40}f(q^3,q^5)f(Q,Q^7),
\end{align*}
where in the last step we used \eqref{trsdis}. We then replace $q$ by $q^2$ and use \eqref{pdis} to conclude that
\begin{equation}\label{dfe8}
\begin{split}
2(4,3,5,q^2) &= \t(q^2)\t(Q^2)+\t(-q^2)\t(-Q^2)+4q^{36}\psi(q^4)\psi(Q^4)\\
             &\quad +2q^9\psi(q)\psi(Q)-2q^9\psi(-q)\psi(-Q).
\end{split}
\end{equation}
In \eqref{atk5}, we use \eqref{trsdis}, then we replace $q$ by $q^2$ and subtract \eqref{dfe8} from the resulting identity to obtain
\begin{equation*}
\begin{split}
2(2,1,9,q^2)-2(4,3,5,q^2) &= \t(q)\t(Q)+\t(-q)\t(-Q)+4q^{18}\psi(q^2)\psi(Q^2)\\
                          &\quad -\t(q^2)\t(Q^2)-\t(-q)\t(-Q)-4q^{36}\psi(q^4)\psi(Q^4)\\
                          &\quad -2q^9\psi(q)\psi(Q)+2q^9\psi(-q)\psi(-Q).
\end{split}
\end{equation*}
From the identities established in \cite[pp.~448--9]{III} it now follows that
\begin{equation}\label{njf}
(2,1,9,q^2)-(4,3,5,q^2)=q^3E(-q)E(-Q)-q^3E(q)E(Q)-2q^{12}E(q^4)E(Q^4).
\end{equation}
By \eqref{tkmm1} and \eqref{tkmm3}, with $q$ replaced by $q^2$ in each, and by \eqref{njf}, we conclude that
\begin{equation*}
2q^4E(q^2)E(Q^2)U(1,71,q^2)=q^3E(-q)E(-Q)-q^3E(q)E(Q)-2q^{12}E(q^4)E(Q^4)-2q^6E(q^2)E(Q^2),
\end{equation*}
which is clearly equivalent to \eqref{tkmm4}.

Lastly, the identity \eqref{tkmm5} follows trivially from \eqref{tkmm1}--\eqref{tkmm3}.
\end{proof}
\begin{remark}
Ramanujan almost always expressed the function $U(r,s)$ in terms of the function $\k(q)$ at related arguments, but he did not have an identity for the modulus $71$.
\end{remark}

The following identity  is for the quadratic forms of discriminant $-56$. This identity was stated in \cite[p.~25]{bh} without a proof. Note that we now set $Q:=q^7$.
\begin{theorem} Let $Q:=q^7$. Then,
\begin{equation}\label{rmc1}
\df{(1,0,14)-(3,2,5)}{(2,0,7)+(3,2,5)}=q\df{E^2(Q^4)E^2(q^2)}{E^2(Q^2)E^2(q^4)}.
\end{equation}
\end{theorem}
\begin{proof}
From \eqref{thr3p1} we have
\begin{equation}\label{rmc2}
4qE(q^2)E(Q^4)U(1,56)U(4,14)=2(1,0,14)-T_5(1,0,14)=2(1,0,14)-2(3,2,5).
\end{equation}
Employing \eqref{thr3p2}, we also find that
\begin{equation}\label{rmc3}
4E(q^4)E(Q^2)U(2,28)U(7,8)=2(2,0,7)+T_5(2,0,7)=2(2,0,7)+2(3,2,5).
\end{equation}
Ramanujan observed that \cite{rogers}
\begin{equation}\label{rmc4}
\df{U(1,14)}{U(2,7)}=\df{E^2(q^2)E^2(Q)}{E^2(q)E^2(Q^2)}.
\end{equation}
From \cite[Th.~1.2]{hb} we get
\begin{equation}\label{rmc5}
\df{U(1,56)}{U(7,8)}=\df{E(q^4)E(Q^2)}{E(q^2)E(Q^4)}.
\end{equation}
The identity \eqref{rmc1} now follows from \eqref{rmc2}--\eqref{rmc5}, where \eqref{rmc4} is used with $q$ replaced by $q^2$.
\end{proof}

The next identity concerns quadratic forms of discriminant $-224$.
\begin{theorem} Let $Q:=q^7$. Then,
\begin{equation}\label{rcmm1}
\df{(1,0,56)-(5,4,12)}{(7,0,8)+(3,2,19)}=q\df{E(q^4)E(Q^8)}{E(q^8)E(Q^4)}.
\end{equation}
\end{theorem}
\begin{proof}
The proof of \eqref{rcmm1} is very similar to that of \eqref{rmc1}, so we omit the details. Identities similar to \eqref{rmc2} and \eqref{rmc3} are established for $U(1,56)U(1,56,q^4)$ and $U(7,8)U(7,8,q^4)$ by using \eqref{thr3p2} and \eqref{thr3p1}, and then the identity \eqref{rmc4} is used twice with $q$ replaced by $q^4$ in its second application.
\end{proof}

By using Ramanujan's identities \cite{db}
\begin{equation}\label{vc1}
\df{U(1,54)}{U(2,27)}=\df{E(q^{27})E(q^{18})E(q^3)E(q^2)}{E(q^{54})E(q^9)E(q^6)E(q)}, \quad \df{U(1,34)}{U(2,17)}=\df{\k(-q^{17})}{\k(-q)}, \quad \text{and} \quad U(2,13)=U(1,26),
\end{equation}
and his other identities \cite{tony}
\begin{equation}\label{vc2}
\df{U(1,66)}{U(2,33)}=\df{E(q^{11})E(q^6)}{E(q^{22})E(q^3)} \quad \text{and} \quad \df{U(3,22)}{U(6,11)}=\df{E(q^2)E(q^{33})}{E(q)E(q^{66})},
\end{equation}
and by following exactly the same arguments as in the previous proof, we can easily establish the following identities, in corresponding order to the identities in \eqref{vc1} and \eqref{vc2}, involving quadratic forms of discriminants $-216$, $-136$, $-104$, and $-264$.

\begin{theorem} The following five facts are true :
\begin{equation*}
\df{(2,0,27)-(7,6,9)}{(1,0,54)+(5,2,11)}=q^2\df{\psi(-q)\psi(-q^6)\psi(-q^9)\psi(-q^{54})}{E(q^3)E(q^{72})\psi(-q^2)\psi(-q^{27})},
\end{equation*}
\begin{equation*}
\df{(2,0,17)-(5,2,7)}{(1,0,34)+(5,2,7)}=q^2\df{\t(-q^4)\psi(q^{17})}{\t(-q^{68})\psi(q)},
\end{equation*}
\begin{equation*}
\df{(1,0,26)-(5,4,6)}{(2,0,13)+(3,2,9)}=q\df{E(q^2)E(q^{52})}{E(q^4)E(q^{26})},
\end{equation*}
\begin{equation*}
\df{(6,0,11)-(7,4,10)}{(3,0,22)+(5,4,14)}=q^6\df{E(q^{12})E(q^{22})\psi(-q)\psi(-q^{66})}{E(q^6)E(q^8)E(q^{33})E(q^{44})},
\end{equation*}
\begin{equation*}
\df{(1,0,66)-(5,4,14)}{(2,0,33)+(7,4,10)}=q\df{E(q^{132})E(q^{44})E(q^{24})E(q^{11})E(q^6)E(q^2)}{E(q^{88})E(q^{66})E(q^{22})E(q^{12})E(q^{4})E(q^3)}.
\end{equation*}
\end{theorem}

The following relations are for quadratic forms of discriminant $-1664$. Here we set $Q:=q^{13}$.

\begin{theorem} The following facts are true (with $Q:=q^{13}$):
\begin{equation}\label{gzi1}
2q^3E(q^{16})E(Q^8)U(2,13,q^8)=(3,2,139)-(12,4,35),
\end{equation}
\begin{equation}
2q^{7}E(q^8)E(Q^{16})U(1,26,q^8)=(7,4,60)-(15,4,28),\label{gzi2}\\
\end{equation}
\begin{equation}
2q^{9}E(q^8)E(Q^{16})=(9,8,48)-(17,6,25),\label{gz1}\\
\end{equation}
\begin{equation}
2q^5E(q^{16})E(Q^8)=(5,4,84)-(20,4,21),\label{gz2}\\
\end{equation}
\begin{subequations}
\begin{align}
2q^5E(q^8)E(Q^{16})U(1,26,q^8)^2 &= (5,4,84)+(21,10,21)-(13,0,32)-(20,4,21)\label{gz3}\\
&=2q^5E(q^{16})E(Q^{8})-2q^{13}\psi(Q^8)\t(-q^8),\label{gz3b}
\end{align}
\end{subequations}
\begin{subequations}
\begin{align}
2qE(q^{16})E(Q^8)U(2,13,q^8)^2 &= (1,0,416)+(17,6,25)-(4,4,105)-(9,8,48)\label{gz4}\\
&=2q\psi(q^8)\t(-Q^8)-2q^9E(q^8)E(Q^{16}),\label{gz4b}
\end{align}
\end{subequations}
\begin{equation}\label{gz5}
\begin{split}
\df{(5,4,84)+(21,10,21)-(13,0,32)-(20,4,21)}{(1,0,416)+(17,6,25)-(4,4,105)-(9,8,48)} &= \df{(9,8,48)-(17,6,25)}{(5,4,84)-(20,4,21)}\\
&= \df{(7,4,60)-(15,4,28)}{(3,2,139)-(12,4,35)}\\
&= q^4\df{E(q^8)E(Q^{16})}{E(q^{16})E(Q^8)}.
\end{split}
\end{equation}
\end{theorem}
\begin{proof}
The derivations of \eqref{gzi1}, \eqref{gzi2}, \eqref{gz1}, and \eqref{gz2} are similar to that of \eqref{tk1}. The identity \eqref{gz3} is obtained from \eqref{gz1} by arguing as in the proof of \eqref{tkmm2}. This reasoning also applies to obtaining \eqref{gz4} from \eqref{gz2}. The last identity, \eqref{gz5}, follows from \eqref{gz1}, \eqref{gz2}, \eqref{gz3}, and \eqref{gz4}, together with the right-most identity in \eqref{vc1}  with $q$ replaced by $q^8$. By \eqref{gz1}, the identity \eqref{gz4b} reduces to $(1,0,416)-(4,4,105)=2q\psi(q^8)\t(-Q^8)$. From \eqref{qrr}, \eqref{prt1}, and \eqref{RGdef}, we have
\begin{align*}
(4,4,105)
&= (105,4,4)\\
&= R(0,0,0,0,2,832,2,4,210)\\
&= R(0,0,0,0,4,416,1,2,210)\\
&= \t(q^4)\t(Q^{32})+4q^{105}\psi(q^8)\psi(Q^{64}).
\end{align*}
If we apply \eqref{tdis} twice, with $q$ replaced by $-Q^8$ in the second application, we can conclude that
\begin{align*}
(1,0,416)-(4,4,105)
&= \t(q)\t(Q^{32})-\t(q^4)\t(Q^{32})-4q^{105}\psi(q^8)\psi(Q^{64})\\
&= \t(Q^{32})\left(\t(q)-\t(q^4)\right)-4q^{105}\psi(q^8)\psi(Q^{64})\\
&= 2q\t(Q^{32})\psi(q^8)-4q^{105}\psi(q^8)\psi(Q^{64})\\
&= 2q\psi(q^8)\left(\t(Q^{32})-2q^{104}\psi(Q^{64})\right)\\
&= 2q\psi(q^8)\t(-Q^8),
\end{align*}
thus proving \eqref{gz4b}. The proof of \eqref{gz3b} is similar and so we skip its details.
\end{proof}

\begin{remark}
Replacing $q^8$ with $q$ in \eqref{gz3b} and \eqref{gz4b} yields Ramanujan's identity \cite{tony}
\begin{equation*}
U(1,26)=U(2,13)=\sqrt{\df{\k(-Q)}{\k(-q)}-q\df{\k(-q)}{\k(-Q)}}.
\end{equation*}
\end{remark}

Next, we obtain relations for the quadratic forms $(3,0,7)$, $(1,0,21)$, $(5,4,5)$, and $(2,2,11)$ of discriminant $-84$. Here we set $Q:=q^7$.
\begin{theorem}The following two facts are true (with $Q:=q^7$):
\begin{equation}\label{nd1}
\df{(1,0,21)-(5,4,5)}{(3,0,7)+(2,2,11)}=q\df{E(q^2)E(Q^6)}{E(q^6)E(Q^2)},
\end{equation}
\begin{equation}\label{nd2}
\df{(3,0,7)-(2,2,11)}{(1,0,21)+(5,4,5)}=-q^2\df{E(q^6)E(Q^2)\psi^2(-q)\psi^2(-Q^3)}{E(q^2)E(Q^6)\psi^2(-q^3)\psi^2(-Q)}.
\end{equation}
\end{theorem}

\begin{proof}
In \cite{hb}, we obtained identities relating the functions $U(7,12)$, $U(4,21)$, $U(3,28)$, and $U(1,84)$. While it was not stated there it trivially follows from
\cite[eqs.\ 4.2, 1.8, 4.28, 4.50]{hb} that
\begin{equation*}
\df{U(7,12)}{U(4,21)}=\df{\psi(-q)\psi(-Q^3)}{\psi(-q^3)\psi(-Q)}.
\end{equation*}
Also from \cite[eqs.\ 1.5, 1.6, 1.9, 4.39, 4.40, 4.42, 4.51]{hb} we find that
\begin{equation*}
\df{U(3,28)}{U(1,84)}=\df{\psi(-q)\psi(-Q^3)}{\psi(-q^3)\psi(-Q)}.
\end{equation*}
Together that is
\begin{equation}\label{nd3}
\df{U(7,12)}{U(4,21)}=\df{U(3,28)}{U(1,84)}=\df{\psi(-q)\psi(-Q^3)}{\psi(-q^3)\psi(-Q)}.
\end{equation}
Next, we have the following from \eqref{thr1} and \eqref{thr2}:
\begin{equation}\label{nd4}
4E(q^6)E(Q^2)U(3,7)U(3,7,q^4)=2(3,0,7)+T_5(3,0,7)=2(3,0,7)+2(2,2,11),
\end{equation}
\begin{equation}
\begin{split}
4q^3E(q^6)E(Q^2)U(12,7)U(3,28) &= -4q^2E(q^6)E(Q^2)U(7,12)U(3,28)\\
                               &= 2(3,0,7)-T_5(3,0,7)=2(3,0,7)-2(2,2,11),
\end{split}
\end{equation}
\begin{equation}
4qE(q^2)E(Q^6)U(1,21)U(1,21,q^4)=2(1,0,21)-T_5(1,0,21)=2(1,0,21)-2(5,4,5),
\end{equation}
\begin{equation}\label{nd7}
4E(q^2)E(Q^6)U(4,21)U(1,84)=2(1,0,21)+T_5(1,0,21)=2(1,0,21)+2(5,4,5).
\end{equation}
Finally, \eqref{nd1} and \eqref{nd2} follow from \eqref{nd3}, \eqref{nd4}--\eqref{nd7}, and Ramanujan's identity \cite{tony}:
\begin{equation*}
U(3,7)=U(1,21).
\end{equation*}
\end{proof}
For our last application, we treat discriminant $-76$. Here we set $Q:=q^{19}$ and we provide identities similar to those that Ramanujan gave for his functions.

\begin{theorem} The following two facts are true (with $Q:=q^{19}$):
\begin{equation}\label{abt1}
4U(1,19,q)U(1,19,q^4)=3\k(q)^2\k(Q)^2+\k(-q)^2\k(-Q)^2+4q^5\df{1}{\k(-q^2)^2\k(-Q^2)^2},
\end{equation}
\begin{equation}\label{abt2}
4qU(1,19)U(1,19,-q)=\k(q)^2\k(Q)^2-\k(-q)^2\k(-Q)^2+12q^5\df{1}{\k(-q^2)^2\k(-Q^2)^2}.
\end{equation}
\end{theorem}

\begin{proof}
Using \eqref{thr1}, we find that
\begin{align}
4E(q^2)E(Q^2)U(1,19)U(1,19,q^4)=2(1,0,19)+T_5(1,0,19)=2(1,0,19)+2(4,2,5),\label{abt3}\\
\intertext{and}
4qE(q^2)E(Q^2)U(4,19)U(1,76)=2(1,0,19)-T_5(1,0,19)=2(1,0,19)-2(4,2,5).\label{abt4}
\end{align}
From \cite{hb} we know that
\begin{equation}\label{abt5}
U(4,19)U(1,76)=U(1,19,q^2).
\end{equation}
Ramanujan observed that \cite{db}
\begin{align}
4qU(1,19,q^2)&=\df{\t(q)\t(Q)-\t(-q)\t(-Q)-4q^5\psi(q^2)\psi(Q^2)}{E(q^2)E(Q^2)}\label{abt5a}\\
&=\k(q)^2\k(Q)^2-\k(-q)^2\k(-Q)^2-4q^5\df{1}{\k(-q^2)^2\k(-Q^2)^2}.\label{abt6}
\end{align}
By adding the identities in \eqref{abt3} and \eqref{abt4} and by using \eqref{abt5} and \eqref{abt6} we arrive at \eqref{abt1}.
From \eqref{abt5a}, \eqref{abt5}, and \eqref{abt2}, we conclude that
\begin{equation}\label{abt8}
2(4,2,5)=\t(q)\t(Q)+\t(-q)\t(-Q)+4q^5\psi(q^2)\psi(Q^2).
\end{equation}
If we use \eqref{abt8} and equate the even parts in both sides of \eqref{abt3}, then by arguing as in the proof of \eqref{ap2}, we will arrive at \eqref{abt2}.
\end{proof}
\section{conclusion}
There are similar identities for many other discriminants that can be proved by establishing identities for the relevant Rogers--Ramanujan functions. For example, for quadratic forms of discriminant $-111$, we find that
\begin{equation}\label{cnj1}
 \df{(4,1,7)-(5,3,6)}{(1,1,28)-(4,1,7)}=\df{(3,3,10)-(4,1,7)}{(2,1,14)+(4,1,7)}=q^3\df{E(q)E(q^{111})}{E(q^3)E(q^{37})}.
 \end{equation}
Another set of examples is for quadratic forms of discriminant $-119$:
\begin{equation}
\df{(4,3,8)-(6,5,6)}{(1,1,30)+(5,1,6)}=q^4\df{E(q)E(q^{119})}{E(q^7)E(q^{17})}
\end{equation}
\begin{equation}
\left((4,3,8)-(2,1,15)\right)\left((3,1,10)-(5,1,6)\right)=\left((1,1,30)+(5,1,6)\right)\left((6,5,6)-(5,1,6)\right).
\end{equation}

There are further identities for quadratic forms not related to Rogers--Ramanujan functions. As an example, for quadratic forms of discriminant $-80$ we have
\begin{equation}\label{cnj4}
\df{(1,0,20)-(3,2,7)}{(3,2,7)+(4,0,5)}=q\df{E(q^{40})E(q^2)}{E(q^{10})E(q^8)}.
\end{equation}

The identities \eqref{cnj1}--- \eqref{cnj4}, along with similar types of identities, will be discussed elsewhere.

\section{Acknowledgment}
We would like to thank Rainer Schulze-Pillot and Keith Grizzell for their interest and helpful comments.


\begin{thebibliography}{99}

\bibitem{bh}
A.~Berkovich and H.~Yesilyurt,
\emph{Ramanujan's identities and representation of integers by certain binary and quaternary quadratic forms},
Ramanujan J.\ \textbf{20} (2009), no.~3, 375--408.


\bibitem{III}
B.~C.~Berndt,
\emph{Ramanujan's Notebooks, Part III},
Springer-Verlag, New York, 1991.


\bibitem{hb}
B.~C.~ Berndt and H.~Yesilyurt,
\emph{New identities for Rogers--Ramanujan functions}
Acta Arith.\ \textbf{120} (2005), no.~4, 395--413.

\bibitem{memoir}
B.~C.~Berndt, G.~Choi, Y.-S.~Choi, H.~Hahn, B.~P.~Yeap, A.~J.~Yee, H.~Yesilyurt, and J.~Yi,
\emph{Ramanujan's forty identities for the Rogers--Ramanujan functions},
Mem.\ Amer.\ Math.\ Soc.\ \textbf{188} (2007), no.~880.


\bibitem{tony}
A.~J.~F.~Biagioli,
\emph{A proof of some identities of Ramanujan using modular forms},
Glasgow Math.~J.\ \textbf{31} (1989), 271--295.


\bibitem{db}
D.~Bressoud,
\emph{Some identities involving Rogers--Ramanujan-type functions},
J.~London Math.\ Soc.\ (2) \textbf{16} (1977), 9--18.


\bibitem{ppair}
K.~Bringmann and H.~Swisher,
\emph{On a conjecture of Koike on identities between Thompson series and Rogers--Ramanujan functions},
Proc.\ Amer.\ Math.\ Soc.\ \textbf{135} (2007), 2317--2326.


\bibitem{Hecke}
E.~Hecke,
\emph{Mathematische Werke},
Vandenhoeck \& Ruprecht, G\"{o}ttingen, 1970.


\bibitem{koike}
M.~Koike,
\emph{Thompson series and Ramanujan's identities},
in \emph{Galois Theory and Modular Forms},
K.~Hashimoto, K.~Miyake, and H.~Nakamura, eds.,
 Developments in Math. {\bf{11}},
(2004), Kluwer, Norwell, pp. 367-374.



\bibitem{lnb}
S.~Ramanujan,
\emph{The Lost Notebook and Other Unpublished Papers},
Narosa, New Delhi, 1988.


\bibitem{rogers1894}
L.~J.~Rogers,
\emph{Second memoir on the expansion of certain infinite products},
Proc.\ London Math.\ Soc.\ \textbf{25} (1894), 318--343.


\bibitem{rogers}
L.~J.~Rogers,
\emph{On a type of modular relation},
Proc.\ London Math.\ Soc.\ \textbf{19} (1921), 387--397.


\bibitem{watson}
G.~N.~Watson,
\emph{Proof of certain identities in combinatory analysis},
J.~Indian Math.\ Soc.\ \textbf{20} (1933), 57--69.


\bibitem{jmaa}
H.~ Yesilyurt,
\emph{Elementary proofs of some identities of Ramanujan for the Rogers--Ramanujan functions},
J.~of Math.\ Anal.\  Appl.\ \textbf{388} (2012), no.~1, 420--434.


\bibitem{mjnt}
H.~Yesilyurt,
\emph{A Generalization of a Modular Identity of Rogers},
J.~Number Theory \textbf{129} (2009), no.~6, 1256--1271.


\end{thebibliography}
\end{document}